\documentclass[11pt,a4paper]{article}

\usepackage[T1]{fontenc}
\usepackage[utf8]{inputenc}
\usepackage{lmodern}
\usepackage{xcolor}
\usepackage[margin=27mm]{geometry}
\usepackage{microtype}
\usepackage{hyperref}

\usepackage{amsmath, amsfonts, amssymb, amsthm,mathtools}
\usepackage{graphicx, mathrsfs, dsfont}
\usepackage{tikz}
\usetikzlibrary{calc,arrows.meta,decorations.pathreplacing}
\usepackage{enumitem}
\usepackage{booktabs}
\usepackage{array}
\usepackage[numbers,sort&compress]{natbib}

\newtheorem{theorem}{Theorem}[section]
\newtheorem{proposition}[theorem]{Proposition}
\newtheorem{lemma}[theorem]{Lemma}
\newtheorem{corollary}[theorem]{Corollary}
\theoremstyle{definition}
\newtheorem{definition}[theorem]{Definition}
\newtheorem{example}[theorem]{Example}
\theoremstyle{remark}

\numberwithin{equation}{section}
\setlist[enumerate,1]{label=\textup{(\roman*)},leftmargin=*,itemsep=3pt}

\DeclareMathOperator{\gph}{gph}
\DeclareMathOperator{\rg}{rg}
\newcommand{\R}{\mathbb R}
\newcommand{\N}{\mathbb N}

\title{Poly and Raby's theorem splits in infinite dimensions}
\author{David Salas}
\date{}

\begin{document}
\maketitle

\begin{abstract}
\noindent A classic result of Poly and Raby \cite{Poly-Raby:1984} states that, in finite-dimensional Euclidean spaces, $\mathcal{C}^k$-smoothness (with $k\geq2$) of the squared distance function near a point of a closed set characterizes smoothness of the set as a submanifold, with the same order of differentiability. In this work, we show that this characterization splits in infinite-dimensional Hilbert spaces: $\mathcal{C}^k$-smoothness of the squared distance function characterizes weakly $\mathcal{C}^k$-submanifolds, while $\mathcal{C}^k$-submanifolds are characterized by this smoothness together with an additional pointwise equicontinuity condition on the highest-order derivative of the squared distance. The diffeomorphism of Poly and Raby, its associated graph representation, and the description of its inverse remain the common geometric basis of both characterizations. The development of this work was assisted by GPT-6 Astra.
\end{abstract}


\section{Introduction}\label{sec:intro}

The study of the relation between the smoothness of a set and that of its distance function can be traced back to Nash \cite{Nash:1952}, who characterized analytic submanifolds in finite dimensions through the local analyticity of their metric projections.

The differentiability of the distance function was then studied in depth in the setting of infinite-dimensional convex analysis, starting with \cite{Zarantonello:1971}, with the objective of understanding the first-order variation of the metric projection onto convex sets. Shortly after, the seminal work of R.B. Holmes \cite{Holmes:1973} showed, already in infinite-dimensional Hilbert spaces, that for a convex body, the $\mathcal{C}^{k}$-smoothness of the boundary of the set, for $k\geq 2$, entails the $\mathcal{C}^{k-1}$-smoothness of the metric projection on the exterior of the body. Afterwards, in 1982, Fitzpatrick and Phelps \cite{Fitzpatrick-Phelps:1984} obtained a converse to Holmes' theorem, characterizing convex bodies with smooth boundary in terms of the smoothness of their metric projection and an invertibility condition on its derivative. The same study was conducted for nonconvex bodies in \cite{Correa-Salas-Thibault:2018,Salas-Thibault:2021}, by replacing convexity by prox-regularity, whose local theory was developed in \cite{Poliquin-Rockafellar-Thibault:2000} and whose origins go back to Federer's sets with positive reach \cite{Federer:1959}.

Returning to the original line initiated in \cite{Nash:1952}, Poly and Raby \cite{Poly-Raby:1984} proved in 1984 the following result: For $k\geq2$, a closed set $M$ in a finite-dimensional Euclidean space is a $\mathcal{C}^k$-submanifold near a point $\bar{x}\in M$ if and only if the squared distance function $d_M^2(\cdot)$ is of class $\mathcal{C}^k$ near $\bar{x}$. While this result looks similar to the previous discussion, it is structurally very different: The aforementioned works \cite{Correa-Salas-Thibault:2018,Holmes:1973,Fitzpatrick-Phelps:1984,Salas-Thibault:2021} deal with manifolds of codimension 1, as they consider the boundary of sets with nonempty interior (convex or prox-regular), while Poly and Raby deal directly with arbitrary submanifolds. An important open question in this line of research was whether the theorem of Poly and Raby could be extended to infinite dimensions. We refer the reader to \cite[Chapter 17]{Thibault:2023} for a survey of the differentiability of distance functions and metric projections.

This work is devoted to showing that the characterization of Poly and Raby splits in infinite-dimensional Hilbert spaces. On the one hand, the $\mathcal{C}^k$-smoothness of $d_M^2$ characterizes weakly $\mathcal{C}^k$-submanifolds, that is, sets admitting a local graph representation by a weakly $\mathcal{C}^k$ mapping; see Definition~\ref{def:weak-submanifold}. On the other hand, $\mathcal{C}^k$-submanifolds are characterized by the same smoothness of $d_M^2$ together with an additional pointwise equicontinuity condition in the normal coordinates of Poly and Raby. A counterexample shows that the extra condition cannot be omitted, even when the manifold is a curve.

The paper is organized as follows. Section~\ref{sec:preliminaries} recalls the weak differentiability notion used in this work, following \cite{Bachir-Lancien:2003,Gutierrez-Llavona:1993}, and introduces the notation and submanifold definitions. Section~\ref{sec:PolyRaby-coordinates} presents the diffeomorphism of Poly and Raby and the identities used throughout the paper inherited from \cite{Poly-Raby:1984}. Section~\ref{sec:Counterexample} gives the counterexample, showing that $\mathcal{C}^k$-smoothness of the squared distance function is not enough to recover $\mathcal{C}^k$-smoothness of the submanifold. Section~\ref{sec:PolyRaby-characterization} establishes the joint regularity needed for the characterization of weakly $\mathcal{C}^k$-submanifolds and proves that characterization. Section~\ref{sec:Pure-submanifolds} gives the pointwise equicontinuity criterion and the characterization of $\mathcal{C}^k$-submanifolds. Finally, Section~\ref{sec:conclusions} summarizes the results.

% ===============================
% ===============================
% ===============================
% ===============================

\section{Preliminaries}\label{sec:preliminaries}

We adopt the usual notation of Functional Analysis given in \cite{Brezis:2011}. Throughout the paper, $\mathcal{H}$ denotes a real Hilbert space, with inner product $\langle\cdot,\cdot\rangle$ and associated norm $\|\cdot\|$. Given Banach spaces $X$ and $Y$, we denote by $\mathcal{L}(X,Y)$ the space of continuous linear mappings from $X$ to $Y$, and simply by $\mathcal{L}(X)$ when $X=Y$. We write $X^*$ for the dual space of $X$. For $j\geq 1$, $\mathcal{L}^j(X;Y)$ denotes the space of continuous $j$-linear mappings from $X^j$ to $Y$, endowed with its usual operator norm; we set $\mathcal{L}^0(X;Y)=Y$. For an operator $T\in\mathcal{L}(X,Y)$, we denote by  $\rg(T)$ its range, $\mathrm{ker}(T)$ its kernel, and $T^*$ its adjoint. If $O\subset X$ is open and $f:O\to Y$, we denote by $Df(x)$ its Fr\'echet derivative at $x$, whenever it exists, and by $D^jf(x)$ its $j$th Fr\'echet derivative. For mappings of one real variable, we also use the notation $f^{(j)}$. As usual, $\mathcal{C}^k(O,Y)$ denotes the class of mappings whose Fr\'echet derivatives up to order $k$ exist and are continuous. We write $\mathcal{C}^{k,0}$ and $\mathcal{C}^{k,1}$ when the $k$th derivative is, respectively, locally uniformly continuous and locally Lipschitz, as in \cite{Salas-Thibault:2019}.

We use the following notation for partial derivatives: Let $X,Y,E$ be Banach spaces, let $U\subset X$ be open, and let $G\in\mathcal{C}^p(U\times Y,E)$, with $p\geq1$. For $0\leq j\leq p$, we write
\begin{equation}\label{eq:pure-first-partial-definition}
D_1^jG(u,y):=D^j\bigl(G(\cdot,y)\bigr)(u)\in\mathcal{L}^j(X;E),
\end{equation}
with $D_1^0G=G$. That is, $D_1^jG$ denotes the $j$th derivative with respect to the first variable, with $y$ fixed. Under the natural identification of spaces of multilinear mappings, we have that
\[
D_1^jG=D_1\bigl(D_1^{j-1}G\bigr),\qquad 1\leq j\leq p.
\]
For a nonempty closed set $M\subset\mathcal{H}$, we denote by $d_M(x):=\inf_{m\in M}\|x-m\|$ its distance function and by $\operatorname{Proj}_M(x)$ the set of nearest points of $x$ in $M$. When this set is a singleton, its element is called metric projection and it is denoted by $P_M(x)$. We recall that $d_M^2$ is Fr\'echet differentiable on an open set $U$ if and only if $P_M$ is well defined and continuous on $U$; in this case,
\begin{equation}\label{eq:distance-gradient}
\nabla\!\left(\frac12d_M^2\right)(x)=x-P_M(x),\qquad \forall x\in U;
\end{equation}
see, e.g., \cite[Lemma~5(e)]{Colombo-Thibault:2010}. A closed set $M$ is prox-regular at $\bar{x}\in M$ if and only if $P_M$ is single-valued and Lipschitz on a neighborhood of $\bar{x}$; see \cite[Theorem~2.2]{Correa-Salas-Thibault:2018}. For a closed subspace $Z\subset\mathcal{H}$, we write $\pi_Z$ for the orthogonal projection onto $Z$. We use the Hilbert product $Z\times Z^\perp$, endowed with the norm $\|(z,w)\|^2=\|z\|^2+\|w\|^2$, and identify it with $\mathcal{H}$ through the canonical linear isometry
\begin{align*}
L:\mathcal{H}&\to Z\times Z^\perp\\
x&\mapsto L(x)=(\pi_Z(x),\pi_{Z^\perp}(x)).
\end{align*}
Once this identification is made, we write points of $\mathcal{H}$ as pairs $(z,w)$ and use $\pi_Z$ and $\pi_{Z^\perp}$ also for the coordinate projections. All graph representations and normal-coordinate formulas below are written in this product convention.

We recall the notion of submanifold of a (possibly infinite-dimensional) Hilbert space used throughout the paper; see, e.g., \cite{Abraham-Marsden-Ratiu:1988}. For $k\geq 1$, a subset $M$ of $\mathcal{H}$ is a $\mathcal{C}^k$-submanifold near $\bar{x}\in M$ if there exist an open neighborhood $U$ of $\bar{x}$, a closed subspace $E$ of $\mathcal{H}$, and a $\mathcal{C}^k$-diffeomorphism $\Phi:U\to\Phi(U)\subset\mathcal{H}$ such that
\[
\Phi(\bar{x})=0,\qquad \Phi(M\cap U)=E\cap\Phi(U).
\]
The space $E$ is called a model space for $M$ near $\bar{x}$. For $x\in U\cap M$, we write $T_{x}M = D\Phi(x)^{-1}(E)$ and $N_{x}M = (T_xM)^{\perp}$ to denote the tangent space and normal space at $x$, respectively.

Equivalently, after translating $\bar{x}$ to $0$, the set $M$ is a $\mathcal{C}^k$-submanifold near $0$ if and only if there exist a closed subspace $Z\subset\mathcal{H}$, an open neighborhood $V$ of $0$ in $Z$, and a $\mathcal{C}^k$ mapping $\theta:V\to Z^\perp$ such that $M$ coincides locally with $\gph(\theta)$, with $\theta(0)=0$ and $D\theta(0)=0$. In this representation, $Z$ is necessarily the tangent space $T_0M$; see, e.g., \cite[Theorem~2.1]{Salas-Thibault:2019}.

Similarly, $M$ is a $\mathcal{C}^{k,0}$-submanifold (respectively, a $\mathcal{C}^{k,1}$-submanifold) near $\bar{x}$ if the local graph representative $\theta$ can be chosen of class $\mathcal{C}^{k,0}$ (respectively, $\mathcal{C}^{k,1}$); see \cite[Section~4]{Salas-Thibault:2019}.

We next recall the weak differentiability notion used in this work, following \cite{Bachir-Lancien:2003}.

\begin{definition}[Weakly Fréchet differentiability]
Let $X,Y$ be Banach spaces, $O\subset X$ be open, and $f:O\to Y$. We say that $f$ is \emph{weakly Fr\'echet differentiable} at $x\in O$ if $y^*\circ f$ is Fr\'echet differentiable at $x$ for every $y^*\in Y^*$. Likewise, for an integer $k\geq 1$, we say that $f$ is \emph{weakly $\mathcal{C}^k$} on $O$ if $y^*\circ f\in\mathcal{C}^k(O,\R)$ for every $y^*\in Y^*$. 
\end{definition}

It is known that weakly $\mathcal{C}^k$ mappings are of class $\mathcal{C}^{k-1}$; see \cite{Gutierrez-Llavona:1993}. Inspired by graph representation and the concept of weakly $\mathcal{C}^k$ functions, we introduce the following graph-based notion of weak regularity for submanifolds.

\begin{definition}[weakly $\mathcal{C}^k$-submanifold]\label{def:weak-submanifold}
Let $k\geq2$, $M\subset\mathcal{H}$, and $\bar{x}\in M$. We say that $M$ is a \emph{weakly $\mathcal{C}^k$-submanifold} near $\bar{x}$ if there exist a closed subspace $Z\subset\mathcal{H}$, an open neighborhood $V$ of $0$ in $Z$, and a weakly $\mathcal{C}^k$ mapping $\theta:V\to Z^\perp$ with $\theta(0)=0$ and $D\theta(0)=0$ such that, after translating $\bar{x}$ to $0$ and identifying $\mathcal{H}$ with $Z\times Z^\perp$, $M$ coincides locally with $\gph(\theta)$.
\end{definition}

% ===============================
% ===============================
% ===============================
% ===============================

\section{The diffeomorphism of Poly and Raby}\label{sec:PolyRaby-coordinates}

The geometric construction behind the characterizations of this work is the coordinate system presented in \cite{Poly-Raby:1984} (which extends directly to Hilbert spaces as developed in \cite[Lemma~5.2 and the proof of Theorem~5.1]{Salas-Thibault:2019}). More precisely, assuming that $\bar{x}=0$ and writing $Z=\operatorname{rg}(DP_M(0))$, the authors considered the mapping
\[
\Psi(x)=\bigl(\pi_Z(P_M(x)),\,\pi_{Z^\perp}(x-P_M(x))\bigr),
\]
which we now call the \emph{Poly--Raby coordinates}. Next proposition summarizes the main properties of the Poly--Raby coordinates, including the graph representation they induce.
% An inverse function argument gives a local diffeomorphism and a graph representation $\theta:V\subset Z\to Z^\perp$ of $M$ with $\theta(0)=0$ and $D\theta(0)=0$. In the product coordinates $Z\times Z^\perp$, the inverse is given by
% \[
% \Psi^{-1}(z,w)=\bigl(z-D\theta(z)^*w,\,\theta(z)+w\bigr),
% \]
% and one has
% \begin{equation}\label{eq:intro-adjoint-identity}
% D\theta(z)^*w=-\frac12\pi_Z\nabla d_M^2\bigl(\Psi^{-1}(z,w)\bigr).
% \end{equation}

%We recall this construction in Section~\ref{sec:PolyRaby-coordinates}. 

% We now study the original construction of Poly and Raby \cite{Poly-Raby:1984}. The diffeomorphism $\Psi$, the graph representation that it induces, and the description of its inverse presented in the Introduction form the geometric basis of the theorem of Poly and Raby \cite{Poly-Raby:1984}. These constructions extend directly to Hilbert spaces, as shown in \cite[Lemma~5.2 and the proof of Theorem~5.1]{Salas-Thibault:2019}. We recall them here in a form that will be used in both directions of our characterization.

\begin{proposition}[Poly--Raby coordinates]\label{prop:PolyRaby-coordinates}
Let $M\subset\mathcal{H}$ be a closed set containing $0$, and let $k\geq2$. Suppose that $d_M^2$ is of class $\mathcal{C}^k$ near $0$. Then $Z:=\rg(DP_M(0))$ is a closed subspace and $DP_M(0)=\pi_Z$. There exist an open neighborhood $U$ of $0$ in $\mathcal{H}$, an open neighborhood $V$ of $0$ in $Z$, and $\delta>0$ such that
\begin{equation}\label{eq:PolyRaby-Psi}
\Psi:U\to V\times B_{Z^\perp}(0,\delta),\qquad
\Psi(x)=\bigl(\pi_Z(P_M(x)),\,\pi_{Z^\perp}(x-P_M(x))\bigr),
\end{equation}
is a $\mathcal{C}^{k-1}$-diffeomorphism. Moreover, the mapping
\begin{equation*}%\label{eq:PolyRaby-graph}
\begin{aligned}
\theta:V&\to Z^\perp\\
z&\mapsto \theta(z):=\pi_{Z^\perp}\Psi^{-1}(z,0),
\end{aligned}
\end{equation*}
is of class $\mathcal{C}^{k-1}$, satisfies $\theta(0)=0$ and $D\theta(0)=0$, and represents $M$ locally as a graph. For every $(z,w)\in V\times B_{Z^\perp}(0,\delta)$, one has
\begin{align}
P_M\bigl(\Psi^{-1}(z,w)\bigr)&=(z,\theta(z)),\label{eq:PolyRaby-projection}\\
\Psi^{-1}(z,w)&=\bigl(z-D\theta(z)^*w,\,\theta(z)+w\bigr),\label{eq:PolyRaby-inverse}\\
D\theta(z)^*w&=-\frac12\pi_Z\nabla d_M^2\bigl(\Psi^{-1}(z,w)\bigr).\label{eq:PolyRaby-adjoint}
\end{align}
\end{proposition}

\begin{proof}
By \eqref{eq:distance-gradient}, $P_M$ is well defined and of class $\mathcal{C}^{k-1}$ near $0$. Differentiating $P_M\circ P_M=P_M$ at $0$ gives $DP_M(0)^2=DP_M(0)$. Furthermore, $DP_M(0)$ is symmetric, since $P_M$ is the gradient of the convex function $x\mapsto\frac12\|x\|^2-\frac12d_M^2(x)$; see \cite{Aslpund:1968}. Hence $DP_M(0)=\pi_Z$ for the closed subspace $Z=\rg(DP_M(0))$.

Consider $\Psi$ as in \eqref{eq:PolyRaby-Psi}, initially on a neighborhood of $0$. Its derivative at $0$ is the canonical isometry $L$ of Section~\ref{sec:preliminaries}, so the Inverse Function Theorem gives a local $\mathcal{C}^{k-1}$-diffeomorphism. Shrink its domain so that it is the inverse image of a product $V\times B_{Z^\perp}(0,\delta)$ and both $x$ and $P_M(x)$ remain in the original injectivity neighborhood. The identity
\[
\Psi(P_M(x))=(\pi_Z(P_M(x)),0)
\]
then gives
\[
P_M\bigl(\Psi^{-1}(z,w)\bigr)=\Psi^{-1}(z,0).
\]
In particular, $\Psi^{-1}(z,0)\in M$ and its first coordinate is $z$. Thus the function $\theta$ gives the asserted graph representation and \eqref{eq:PolyRaby-projection}. The identities $\theta(0)=0$ and $D\theta(0)=0$ follow from $D\Psi(0)=L$.

For a $\mathcal{C}^1$ graph, its normal space, as the orthogonal complement of the tangent space (see, e.g., \cite{Abraham-Marsden-Ratiu:1988}), is given by
\begin{equation}\label{eq:graph-normal-space}
N_{(z,\theta(z))}M
=\{(-D\theta(z)^*\eta,\eta):\eta\in Z^\perp\}.
\end{equation}
Fix $x=\Psi^{-1}(z,w)$. Since $x-P_M(x)$ is normal to $M$ at $P_M(x)$ and its second coordinate is $w$, we obtain
\[
x-P_M(x)=(-D\theta(z)^*w,w).
\]
Together with \eqref{eq:PolyRaby-projection}, this proves \eqref{eq:PolyRaby-inverse}. Finally, \eqref{eq:distance-gradient} gives \eqref{eq:PolyRaby-adjoint}.
\end{proof}

The same coordinates can also be obtained from the graph representation. We separate its geometric argument from the regularity required to apply the Inverse Function Theorem.

\begin{proposition}[Inverse of Poly--Raby coordinates]\label{prop:PolyRaby-from-graph}
Let $M\subset\mathcal{H}$ be closed and coincide near $0$ with the graph of a $\mathcal{C}^1$ mapping $\theta:V\subset Z\to Z^\perp$, with $\theta(0)=0$ and $D\theta(0)=0$. Suppose that $P_M$ is well defined and continuous near $0$ and that, for an integer $r\geq1$, the normal-coordinate mapping
\begin{equation}\label{eq:PolyRaby-psi}
\psi:V\times Z^\perp\to Z\times Z^\perp,\qquad
\psi(z,w)=\bigl(z-D\theta(z)^*w,\,\theta(z)+w\bigr),
\end{equation}
is of class $\mathcal{C}^r$ near $(0,0)$. Then $\psi$ is a local $\mathcal{C}^r$-diffeomorphism, and its inverse is precisely the mapping $\Psi$ in \eqref{eq:PolyRaby-Psi}. In particular,
\begin{equation}\label{eq:PolyRaby-projection-from-graph}
P_M(x)=\bigl(\pi_Z\Psi(x),\,\theta(\pi_Z\Psi(x))\bigr)
\end{equation}
for $x$ close enough to $0$.
\end{proposition}

\begin{proof}
In the product identification, $\psi(0,0)=0$ and $D\psi(0,0)=\operatorname{id}_{Z\times Z^\perp}$. The Inverse Function Theorem therefore gives a local $\mathcal{C}^r$-diffeomorphism. For $x$ close enough to $0$, continuity of $P_M$ allows us to write $P_M(x)=(z,\theta(z))$ for some sufficiently small $z\in Z$. Since $x-P_M(x)$ is normal to $M$ at $(z,\theta(z))$,~\eqref{eq:graph-normal-space} yields that there is $w\in Z^{\perp}$ such that $x-P_M(x) = (-D\theta(z)^*w,w)$. Thus, we can write
\[
x=(z,\theta(z)) + (-D\theta(z)^*w,w) = \psi(z,w),\qquad w=\pi_{Z^\perp}(x-P_M(x)).
\]
Both $z$ and $w$ tend to $0$ as $x\to0$, so they lie in the local injectivity neighborhood of $\psi$. Consequently,
\[
\psi^{-1}(x)=(z,w)
=\bigl(\pi_Z(P_M(x)),\,\pi_{Z^\perp}(x-P_M(x))\bigr)=\Psi(x),
\]
which also proves \eqref{eq:PolyRaby-projection-from-graph}.
\end{proof}

In \cite{Poly-Raby:1984}, the Poly--Raby coordinates are used to show that in the finite-dimensional setting, $M$ is a $\mathcal{C}^k$-submanifold near $0$ if and only if $d_M^2$ is of class $\mathcal{C}^k$. The necessary condition follows directly from Proposition~\ref{prop:PolyRaby-from-graph}, while the sufficiency is more delicate: A compactness argument allows to use \eqref{eq:PolyRaby-adjoint} to derive $\mathcal{C}^{k-1}$-smoothness of $z\mapsto D\theta(z)^*$ from the $\mathcal{C}^{k-1}$-smoothness of $\nabla d_M^2\circ \Psi^{-1}$. Note that this argument also works if the manifold $M$ has finite codimiension. This leads to the following corollary.

\begin{corollary}[{\cite{Poly-Raby:1984}}] Let $M\subset\mathcal{H}$ be a closed set, $k\geq2$, and $\bar{x}\in M$. The following assertions are equivalent:
\begin{enumerate}
\item $d_M^2$ is of class $\mathcal{C}^k$ near $\bar{x}$ and $\mathrm{ker}(DP_M(0))$ has finite dimension.
\item $M$ is a $\mathcal{C}^k$-submanifold near $\bar{x}$ with finite codimension.
\end{enumerate} 
\end{corollary}

The obstruction to extend this corollary in infinite dimensions, as discussed in \cite{Salas-Thibault:2019}. is not the construction of $\Psi$, the graph representation, or the formula for the inverse. It is precisely the  argument to pass from the smoothness of $z\mapsto D\theta(z)^*w$ for each fixed $w$ to the same smoothness of the operator-valued mapping $z\mapsto D\theta(z)^*$.

% In finite dimensions, as in \cite{Poly-Raby:1984}, and more generally when the submanifold has finite codimension, this is not an issue, since it follows directly from a uniform continuity argument. In \cite{Salas-Thibault:2019}, the authors addressed this last step through a local uniform equicontinuity hypothesis and obtained the characterizations
% \begin{equation}\label{eq:characs-SalasThibault}
% \begin{aligned}
%     d_M^2\in\mathcal{C}^{k,0}\text{ near }\bar{x}
%     &\iff M\text{ is a }\mathcal{C}^{k,0}\text{-submanifold near }\bar{x},\\
%     d_M^2\in\mathcal{C}^{k,1}\text{ near }\bar{x}
%     &\iff M\text{ is a }\mathcal{C}^{k,1}\text{-submanifold near }\bar{x},
% \end{aligned}
% \end{equation}
% where the indices $0$ and $1$ stand for a locally uniformly continuous and a locally Lipschitz continuous $k$th derivative, respectively. The case where $d_M^2$ is merely of class $\mathcal{C}^k$ remained open. Our results show that the same geometric argument is tight in this case as well: it yields exactly weak $\mathcal{C}^k$ regularity, and pointwise equicontinuity of the last tangential derivative in \eqref{eq:intro-adjoint-identity} is precisely the additional condition needed to recover $\mathcal{C}^k$ regularity. Here, equicontinuity is required at each point of the graph domain, uniformly over a fixed small normal ball, but not uniformly between arbitrary pairs of points in a fixed neighborhood.

% ===============================
% ===============================
% ===============================
% ===============================

\section{A tight example}\label{sec:Counterexample}

In this section, we show that the obstruction observed in \cite{Salas-Thibault:2019} to extending the theorem of Poly and Raby to infinite dimensions is genuine: for every $k\geq2$, there is a closed set $M\subset\R\times\ell^2$ containing $0$ for which $d_M^2$ is of class $\mathcal{C}^k$ near $0$, although $M$ is not a $\mathcal{C}^k$-submanifold there. Its graph representation is nevertheless weakly $\mathcal{C}^k$, anticipating the characterization in Section~\ref{sec:PolyRaby-characterization}.

\begin{example}\label{ex:Counterexample}
Fix $k\geq 2$, and let $\zeta:\R\to [0,1]$ be a $\mathcal{C}^{\infty}$ function which is equal to $1$ on a neighborhood of $0$ and is supported in the open interval $(-1/4,1/4)$. Consider the function $\varphi:\R\to\R$ given by
\[
\varphi(t) := \zeta(t)\left(\frac{t^{k-1}}{(k-1)!} + \frac{t^k}{k!}\right).
\]
Note that this construction entails that $\varphi^{(k-1)}(0) = \varphi^{(k)}(0) = 1$. Let $r_n = 4^{-n}$ and define the curve $\theta:\R\to \ell^2$ given by
\[
\theta(t) = \sum_{n=1}^{\infty} r_n^k \varphi\left(\frac{t-r_n}{r_n}\right) e_n.
\]
Note that the $n$th summand of the series is supported in the interval $I_n = \left( 3r_n/4,5r_n/4 \right)$, and the family $\{I_n\,:\, n\in\N\}$ is pairwise disjoint. Therefore, $\theta$ is well defined, it vanishes in $(-\infty,0]$, and it is $\mathcal{C}^{\infty}$-smooth in $\R\setminus\{0\}$. Figure~\ref{fig:counterexample} illustrates the construction.\hfill$\Diamond$
\end{example}

\begin{figure}[ht]
\centering
\begin{tikzpicture}[x=1cm,y=1cm,>=Latex,line cap=round,line join=round]
  % Main reference points on the t-axis
  \coordinate (O) at (0,0);
  \coordinate (L1) at (5.90,0);
  \coordinate (R1) at (7.05,0);
  \coordinate (C1) at ($(L1)!0.5!(R1)$);
  \coordinate (L2) at (3.55,0);
  \coordinate (R2) at (4.45,0);
  \coordinate (C2) at ($(L2)!0.5!(R2)$);
  \coordinate (L3) at (1.35,0);
  \coordinate (R3) at (2.05,0);
  \coordinate (C3) at ($(L3)!0.5!(R3)$);

  % Directions of the visible coordinates
  \coordinate (u2) at (-0.82,0.57);
  \coordinate (u3) at (-0.74,-0.52);

  % Axes
  \draw[->,thick] (-1.55,0) -- (8.25,0) node[right] {$t$};
  \draw[->,thick] (O) -- (0,3.1) node[above] {$e_1$};
  \draw[->,thick] (O) -- (-2.85,1.98) node[left] {$e_2$};
  \draw[->,thick] (O) -- (-2.20,-1.55) node[left] {$e_3$};
  \fill (O) circle (2pt) node[below] {$0$};
  \node[below] at (-1.2,0) {$t<0$};

  % Baseline pieces of the graph
  \draw[very thick,blue] (-1.30,0) -- (1.02,0);
  \draw[very thick,blue] (1.02,0) -- (L3);
  \draw[very thick,blue] (R3) -- (L2);
  \draw[very thick,blue] (R2) -- (L1);
  \draw[very thick,blue] (R1) -- (7.95,0);

  % Tiny omitted bumps accumulating at 0
  % \draw[dotted,thick]
  %   (0.22,0.00)
  %     .. controls (0.33,0.08) and (0.45,0.10) .. (0.56,0.00)
  %     .. controls (0.63,0.06) and (0.74,0.08) .. (0.82,0.00)
  %     .. controls (0.88,0.05) and (0.95,0.06) .. (1.00,0.00);
  \node[above,blue] at (0.62,0.27) {$\cdots\cdots$};

  % Smooth bump in the e3-direction (smallest)
  \begin{scope}[shift={(1.35,0)},x={(0.70,0)},y={(-0.55,-0.39)}]
    \draw[very thick,blue] plot[smooth] coordinates {
      (0.00,0.00) (0.10,0.02) (0.22,0.13) (0.36,0.35)
      (0.50,0.55) (0.64,0.35) (0.78,0.13) (0.90,0.02) (1.00,0.00)};
  \end{scope}
  %\draw[densely dashed] (C3) -- ++(-0.30,-0.21);
  % \node[left] at (0.95,-0.56) {$\theta_3(t)e_3$};
   \node[above] at ($(C3)+(0,0.08)$) {$I_3$};
  % \node[below] at ($(C3)+(0,-0.38)$) {$I_3$};

  % Smooth bump in the e2-direction (intermediate)
  \begin{scope}[shift={(3.55,0)},x={(0.90,0)},y={(-0.72,0.50)}]
    \draw[very thick,blue] plot[smooth] coordinates {
      (0.00,0.00) (0.10,0.03) (0.22,0.16) (0.36,0.48)
      (0.50,0.95) (0.64,0.48) (0.78,0.16) (0.90,0.03) (1.00,0.00)};
  \end{scope}
  % \draw[densely dashed] (C2) -- ++(-0.68,0.47);
  % \node[left] at (2.52,1.08) {$\theta_2(t)e_2$};
   \node[below] at ($(C2)+(0,-0.10)$) {$I_2$};
  % \node[above] at ($(C2)+(0,0.42)$) {$I_2$};

  % Smooth bump in the e1-direction (largest)
  \begin{scope}[shift={(5.90,0)},x={(1.15,0)},y={(0,1.0)}]
    \draw[very thick,blue] plot[smooth] coordinates {
      (0.00,0.00) (0.08,0.04) (0.18,0.25) (0.31,0.85)
      (0.42,1.45) (0.50,1.78) (0.58,1.45) (0.69,0.85)
      (0.82,0.25) (0.92,0.04) (1.00,0.00)};
  \end{scope}
 % \draw[densely dashed] (C1) -- ++(0,1.78);
 % \node[right] at (6.86,1.95) {$\theta_1(t)e_1$};
  \node[below] at ($(C1)+(0,-0.10)$) {$I_1$};
%  \node[above] at ($(C1)+(0,0.55)$) {$I_1$};
\end{tikzpicture}
\caption{Schematic view of the curve $t\mapsto (t,\theta(t))$ from Example~\ref{ex:Counterexample} in the coordinates $(t,e_1,e_2,e_3)$. On each support interval $I_n$, only the $e_n$-component is active. The visible amplitudes decrease from $e_1$ to $e_2$ to $e_3$, while the omitted higher modes produce analogous smaller bumps accumulating at $t=0$. The figure illustrates the disjoint supports and decreasing amplitudes; the scalar profiles and relative scales are schematic.}
\label{fig:counterexample}
\end{figure}
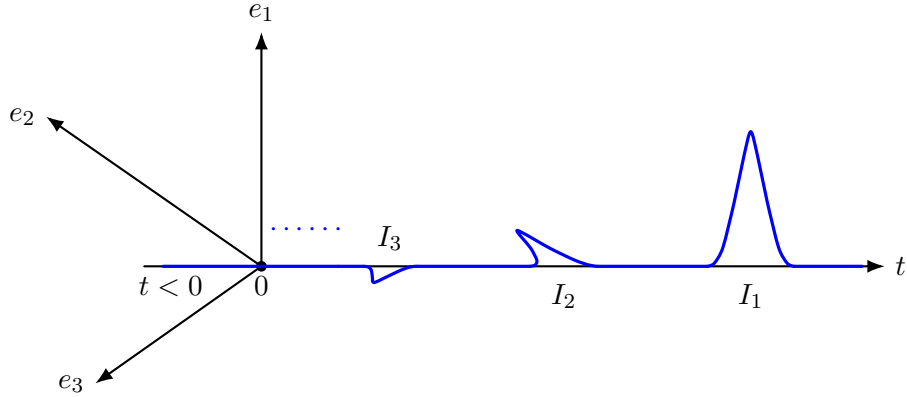
\begin{proposition}\label{prop:Counterexample-properties} Fix $k\geq 2$ and consider the function $\theta:\R\to \ell^2$ given by Example~\ref{ex:Counterexample}. The following assertions hold:
\begin{enumerate}
    \item $\theta$ is of class $\mathcal{C}^{k-1}$.
    \item $\theta^{(k-1)}$ is not Fr\'echet differentiable at $0$.
    \item $\theta$ is weakly $\mathcal{C}^k$.
    \item The squared distance function $d_M^2$ with $M=\gph(\theta)\subset \R\times\ell^2$ is of class $\mathcal{C}^k$ near $0$.
\end{enumerate}
\end{proposition}
\begin{proof}
We prove the assertions sequentially.

\paragraph{(i) $\theta$ is of class $\mathcal{C}^{k-1}$:} For $0\leq j\leq k$, let us define
\[
b_j(t) = \left\{
\begin{array}{cr}
     \displaystyle\sum_{n=1}^{\infty} r_n^{k-j}\varphi^{(j)}\left(\frac{t-r_n}{r_n}\right)e_n&\quad\text{ if }t\neq 0,  \\
     0& \quad\text{ if }t=0.
\end{array}
\right.
\]
We claim that $b_j = \theta^{(j)}$ for every $j\in \{0,\ldots,k-1\}$. It is straightforward to verify that this holds for every $t\neq 0$, and so we only need to prove it for $t=0$.

First note that $b_0=\theta$. For $0\leq j\leq k$, put $C_j:=\|\varphi^{(j)}\|_\infty$. For every $j\in\{0,\ldots,k-1\}$, we have
\[
t\in I_n\implies \|b_j(t)\|\leq C_jr_n^{k-j}
\leq C_j\left(\frac43\right)^{k-j}|t|^{k-j}.
\]
Outside the intervals $I_n$, the function $b_j$ vanishes, so this estimate holds for every $t\in\R$. In particular, $b_j$ is continuous at $0$. Inductively, for $1\leq j\leq k-1$, we obtain
\[
\frac{\|b_{j-1}(t)-b_{j-1}(0)-b_j(0)t\|}{|t|}
\leq C_{j-1}\left(\frac43\right)^{k-j+1}|t|^{k-j}
\xrightarrow{t\to0}0.
\]
This proves our claim and yields $\theta\in\mathcal{C}^{k-1}(\R,\ell^2)$.

\paragraph{(ii) $\theta^{(k-1)}$ is not Fr\'echet differentiable at $0$:} Consider the sequence $t_n=r_n\to 0$. Recalling that $\varphi^{(k-1)}(0) = 1$, we have that
\[
\frac{\|\theta^{(k-1)}(r_n) - \theta^{(k-1)}(0)\|}{r_n} = \|e_n\| = 1.
\]
Since $\theta^{(k-1)}$ vanishes for $t\leq0$, any derivative at $0$ would have to be zero, contradicting this equality. Thus $\theta^{(k-1)}$ is not Fr\'echet differentiable at $0$. Furthermore, $M=\gph(\theta)$ is closed and has tangent space $\R\times\{0\}$ at $0$. If $M$ were a $\mathcal{C}^k$-submanifold there, its graph representation over this tangent space would be of class $\mathcal{C}^k$ by \cite[Theorem~2.1]{Salas-Thibault:2019}. By uniqueness of the local graph representation over the tangent space $\R\times\{0\}$, this would imply $\theta\in\mathcal{C}^k$ near $0$, a contradiction.

\paragraph{(iii) $\theta$ is weakly $\mathcal{C}^k$:} Fix $\xi\in\ell^2$. By (i), it remains to prove that $t\mapsto\langle\theta^{(k-1)}(t),\xi\rangle$ is differentiable at $0$ and that its derivative is continuous there. For $t\in I_n$, we have
\[
\frac{|\langle\theta^{(k-1)}(t)-\theta^{(k-1)}(0),\xi\rangle|}{|t|}
\leq C_{k-1}\frac{r_n}{|t|}|\xi_n|
\leq\frac43 C_{k-1}|\xi_n|.
\]
If $t\to0$ through these intervals, their indices tend to infinity, and $\xi_n\to0$. Outside their union, the numerator is zero. Thus the scalar derivative at $0$ exists and equals zero. Moreover,
\[
t\in I_n\implies |\langle b_k(t),\xi\rangle|\leq C_k|\xi_n|\longrightarrow0
\qquad\text{as }t\to0,
\]
and $b_k(t)=0$ outside the intervals. Hence the $k$th derivative of $\langle\theta(\cdot),\xi\rangle$ is continuous at $0$, and $\theta$ is weakly $\mathcal{C}^k$.

\paragraph{(iv) The squared distance function $d_M^2$ is of class $\mathcal{C}^k$ near $0$:}

For $(s,z)\in \R\times \ell^2$, consider the function
\[
t\mapsto J(t;s,z) :=  \frac{1}{2}(t-s)^2 + \frac{1}{2}\|\theta(t) - z\|^2,
\]
which is a coercive continuous function, and therefore it attains its minimum. Note that $\frac{1}{2}d_M^2(s,z) = \min_{t\in\R} J(t;s,z)$.

The derivative of $J$ with respect to $t$ is given by
\[
\frac{d}{dt}J(t;s,z) = F(t,s,z):= t-s + \langle \theta'(t), \theta(t)-z\rangle.
\]
By Lemma~\ref{lem:joint-adjoint} below, the mapping $(t,y)\mapsto\langle\theta'(t),y\rangle$ is jointly of class $\mathcal{C}^{k-1}$. Composing with $(t,s,z)\mapsto(t,\theta(t)-z)$, we deduce that $F$ is of class $\mathcal{C}^{k-1}$. Noting that $F(0,0,0) = 0$ and $D_tF(0,0,0) = 1$, we can apply the Implicit Function Theorem to show that there is $\delta>0$ small enough, an open neighborhood $O$ of $(0,0)\in\R\times \ell^2$ and a mapping $q\in\mathcal{C}^{k-1}(O,\R)$ such that for every $t\in(-\delta,\delta)$ and every $(s,z)\in O$ one has 
\[
F(t,s,z) = 0 \iff t = q(s,z).
\]

Now fix $(s,z)\in O$ and let $t$ be a global minimizer of $J(\cdot;s,z)$. We have that
\[
(t-s)^2 \leq 2J(t;s,z)\leq 2J(0;s,z) = s^2+\|z\|^2 = \|(s,z)\|^2.
\]
Thus, $|t| \leq |s| + \|(s,z)\| \leq 2\|(s,z)\|$. By shrinking $O$ if necessary, we can assume that $\|(s,z)\| <\delta/2$ and so $|t|<\delta$. Since a global minimizer satisfies $F(t,s,z)=0$, it follows that $t=q(s,z)$. This yields that $P_M$ is well-defined over $O$ and that
\[
\forall (s,z)\in O,\quad P_M(s,z) = (q(s,z),\theta\circ q(s,z)).
\]
Then, $P_M$ is of class $\mathcal{C}^{k-1}$ on $O$, and \eqref{eq:distance-gradient} shows that $d_M^2$ is of class $\mathcal{C}^k$ on $O$. This completes the proof.
\end{proof}

% ===============================
% ===============================
% ===============================
% ===============================

\section{Characterization of weakly \texorpdfstring{$\mathcal{C}^k$}{Ck}-submanifolds}\label{sec:PolyRaby-characterization}

Example~\ref{ex:Counterexample} is not an isolated phenomenon. Its properties describe exactly the regularity characterized by smoothness of the squared distance function. We now establish both implications in a single theorem, using the diffeomorphism of Poly and Raby recalled in Section~\ref{sec:PolyRaby-coordinates}. Before we can prove this equivalence, we need some regularity properties of weakly $\mathcal{C}^k$ functions.

\begin{lemma}\label{lem:joint-adjoint}
Let $Z,Y$ be real Hilbert spaces, let $V\subset Z$ be open, and let $k\geq2$. If $\theta:V\to Y$ is weakly $\mathcal{C}^k$, then the mapping
\[
G:V\times Y\to Z,\qquad G(z_1,z_2):=D\theta(z_1)^*z_2,
\]
is jointly of class $\mathcal{C}^{k-1}$.
\end{lemma}

\begin{proof}
Since $\theta\in\mathcal{C}^{k-1}(V,Y)$, we already have $G\in\mathcal{C}^{k-2}(V\times Y,Z)$. 
For $y\in Y$, write $\theta_y(z):=\langle\theta(z),y\rangle$, which belongs to $\mathcal{C}^k(V,\R)$.

We will first show that the operator $(y,z)\mapsto D^k\theta_y(z)$ is jointly continuous. Note that for every given $z\in V$, the map $B(z):y\mapsto D^k\theta_y(z)$ is linear from $Y$ to $\mathcal{L}^k(Z;\R)$. We need to show that $B(z)$ is also continuous on $y$, for ervery given $z\in V$. Fix $z\in V$ and $h_1,\ldots,h_k\in Z$. For nonzero $t_n\to0$ with $z+t_nh_1\in V$, the continuous linear functionals on $Y$ given by
\[
\ell_n: y\longmapsto\left\langle
\frac{D^{k-1}\theta(z+t_nh_1)-D^{k-1}\theta(z)}{t_n}
(h_2,\ldots,h_k),y\right\rangle
\]
converge pointwise to
\[
\ell_{z,h_1,\ldots,h_k}(y):=D^k\theta_y(z)(h_1,\ldots,h_k).
\]
This yields that $\sup_n |\ell_n(y)|<\infty$ and so, as a direct consequence of the Banach--Steinhaus theorem, each $\ell_{z,h_1,\ldots,h_k}$ is continuous (see, e.g., \cite[Theorem~2.2 and Corollary~2.3]{Brezis:2011}). Moreover, for every $y\in Y$,
\[
\sup_{\|h_i\|\leq1}|\ell_{z,h_1,\ldots,h_k}(y)|=\|D^k\theta_y(z)\|<\infty.
\]
The Banach--Steinhaus theorem, now applied to the family of continuous linear functionals $\{\ell_{z,h_1,\ldots,h_k}:h_i\in Z,\ \|h_i\|\leq1,\ i=1,\ldots,k\}$, yields a constant $C_z$ such that
\begin{equation}\label{eq:scalar-last-bound}
\|D^k\theta_y(z)\|\leq C_z\|y\|,\qquad y\in Y.
\end{equation}
This proves that $B(z):y\mapsto D^k\theta_y(z)$ is continuous for every fixed $z\in V$. Also, since $\theta$ is weakly $\mathcal{C}^k$,
for every fixed $y$ the mapping $z\mapsto B(z)(y)=D^k\theta_y(z)$ is continuous and so $\sup_n \|B(z_n)y\|<+\infty$. Thus, for any sequence $z_n\to z$,  Banach--Steinhaus theorem applied to $\{B(z_n)\mid n\in\N\}$ gives $\sup_n\|B(z_n)\|<\infty$. Consequently, if also $y_n\to y$, then
\[
\|B(z_n)y_n-B(z)y\|
\leq \left(\sup_n\|B(z_n)\|\right)\|y_n-y\|
+\|(B(z_n)-B(z))y\|\longrightarrow0.
\]
Hence $(z,y)\mapsto B(z)(y) = D^k\theta_y(z)$ is jointly continuous as claimed.

Now consider $H(z,y):=D_1^{k-2}G(z,y)\in\mathcal{L}^{k-2}(Z;Z)$, with $H=G$ when $k=2$. Its partial Fr\'echet derivatives therefore satisfy
\begin{align*}
\left\langle\bigl(D_1H(z,y)h_0\bigr)(h_1,\ldots,h_{k-2}),h_{k-1}\right\rangle
&=D^k\theta_y(z)(h_0,h_1,\ldots,h_{k-1}),\\
\left\langle\bigl(D_2H(z,y)\eta\bigr)(h_1,\ldots,h_{k-2}),h_{k-1}\right\rangle
&=\left\langle D^{k-1}\theta(z)(h_1,\ldots,h_{k-1}),\eta\right\rangle.
\end{align*}

The first partial derivative is jointly continuous in operator norm by the joint continuity of $(z,y)\mapsto D^k\theta_y(z)$. The second is jointly continuous in operator norm because $D^{k-1}\theta$ is continuous. Thus $H$ is of class $\mathcal{C}^1$ by the continuous partial derivative criterion.

When $k=2$, this proves the result. When $k>2$, linearity of $G$ in its second variable gives that for $h=(h_1,\ldots,h_{k-2})$ and $\eta_1,\ldots\eta_{k-2}$ one has
\begin{align*}
D^{k-2}G(z,y)\bigl((h_1,\eta_1),\ldots,(h_{k-2},\eta_{k-2})\bigr)=H(z,y)h
+\sum_{i=1}^{k-2} D_1^{k-3}G(z,\eta_i) h_{-i},
\end{align*}
where $h_{-i}$ is the vector $h$ after removing the $i$th coordinate. The coefficients of the terms in the sum are obtained from $D^{k-2}\theta$, which is of class $\mathcal{C}^1$. Together with the $\mathcal{C}^1$ regularity of $H$, this shows that $D^{k-2}G$ is of class $\mathcal{C}^1$, and finishes the proof.
\end{proof}

The previous lemma has the following direct consequence: weakly $\mathcal{C}^k$ functions are not only of class $\mathcal{C}^{k-1}$, as proved in \cite{Gutierrez-Llavona:1993}, but in fact of class $\mathcal{C}^{k-1,1}$.

\begin{corollary}\label{cor:weak-lipschitz}
Under the assumptions of Lemma~\ref{lem:joint-adjoint}, $D^{k-1}\theta$ is locally Lipschitz. In particular, $\theta$ is of class $\mathcal{C}^{1,1}$ when $k\geq2$.
\end{corollary}

\begin{proof}
Denote again $\theta_y(z) = \langle \theta(z),y\rangle$ and $B(z):y\mapsto D^{k}\theta_y(z)$. The preceding proof shows that $B(z)$ is locally bounded: if not, there would exist a convergent sequence $(z_n)$ with $\sup_n\|B(z_n)\| = +\infty$, which is not possible. Choose then a convex neighborhood $W\subset V$ on which $\|B(z)\|\leq C$. For $z,z'\in W$, the mean value estimate for the scalarizations $\theta_y$ gives
\begin{align*}
\|D^{k-1}\theta(z')-D^{k-1}\theta(z)\|
&=\sup_{\|y\|\leq1}\|D^{k-1}\theta_y(z')-D^{k-1}\theta_y(z)\|\leq C\|z'-z\|.
\end{align*}
This completes the proof.
\end{proof}

We are now ready to establish the first generalization of Poly and Raby's theorem in infinite dimensions.

\begin{theorem}\label{thm:weak-characterization}
Let $M\subset\mathcal{H}$ be a closed set, $k\geq2$, and $\bar{x}\in M$. The following assertions are equivalent:
\begin{enumerate}
\item $d_M^2$ is of class $\mathcal{C}^k$ near $\bar{x}$.
\item $M$ is a weakly $\mathcal{C}^k$-submanifold near $\bar{x}$.
\end{enumerate}
\end{theorem}

\begin{proof}
Without loss of generality, assume that $\bar{x}=0$. Both implications rely on the same coordinates $\Psi$ and $\psi=\Psi^{-1}$ from Section~\ref{sec:PolyRaby-coordinates}.

Suppose first that (i) holds. Let $\Psi$, $\theta$, $V$, and $\delta$ be given by Proposition~\ref{prop:PolyRaby-coordinates}. For each fixed $w\in B_{Z^\perp}(0,\delta)$, identity~\eqref{eq:PolyRaby-adjoint} shows that
\[
z\longmapsto D\theta(z)^*w
\]
is of class $\mathcal{C}^{k-1}$ on $V$. By linearity in $w$, the same is true for every $w\in Z^\perp$: for $w\neq0$, use
\[
D\theta(z)^*w=\frac{2\|w\|}{\delta}
D\theta(z)^*\left(\frac{\delta w}{2\|w\|}\right).
\]
Since this mapping is the gradient of $z\mapsto\langle\theta(z),w\rangle$, every scalarization of $\theta$ is of class $\mathcal{C}^k$. Thus $\theta$ is weakly $\mathcal{C}^k$, proving (ii).

Conversely, suppose that (ii) holds, and let $\theta:V\subset Z\to Z^\perp$ be a weakly $\mathcal{C}^k$ graph representation, with $\theta(0)=0$ and $D\theta(0)=0$. Corollary~\ref{cor:weak-lipschitz} gives $\theta\in\mathcal{C}^{1,1}$, which makes $M$ prox-regular at $0$; see, e.g., \cite[Example~64.2]{Colombo-Thibault:2010}. 

By Lemma~\ref{lem:joint-adjoint}, the normal-coordinate mapping $\psi$ in \eqref{eq:PolyRaby-psi} is of class $\mathcal{C}^{k-1}$. Proposition~\ref{prop:PolyRaby-from-graph}, applied with $r=k-1$, identifies its local inverse with the same Poly--Raby mapping $\Psi$ and gives \eqref{eq:PolyRaby-projection-from-graph}. Since $\Psi$ and $\theta$ are of class $\mathcal{C}^{k-1}$, this identity yields $P_M\in\mathcal{C}^{k-1}$ near $0$. Finally, \eqref{eq:distance-gradient} gives (i), finishing the proof.
\end{proof}

% ===============================
% ===============================
% ===============================
% ===============================

\section{Characterization of \texorpdfstring{$\mathcal{C}^k$}{Ck}-submanifolds in Hilbert spaces}\label{sec:Pure-submanifolds}

Theorem~\ref{thm:weak-characterization} gives the first part of the splitting. We now identify the additional condition that recovers strong $\mathcal{C}^k$ regularity. By \eqref{eq:PolyRaby-adjoint}, the issue is to pass from joint regularity of $(z,w)\mapsto D\theta(z)^*w$ to regularity of the operator-valued mapping $z\mapsto D\theta(z)^*$. In \cite[Lemma~5.3]{Salas-Thibault:2019}, local uniform equicontinuity in the first variable, uniformly over a bounded neighborhood of zero in the second variable, of the last derivative was used for this purpose. This is the key observation that leads to \eqref{eq:characs-SalasThibault}. Here, equicontinuity at each point is enough, and is also necessary.

We use the notation $D_1^pG$ fixed in \eqref{eq:pure-first-partial-definition}: only the first variable is differentiated, and all arguments of the resulting multilinear mapping lie in the first-variable space.

\begin{lemma}\label{lem:pointwise-equicontinuity}
Let $X,Y,E$ be Banach spaces, let $U\subset X$ be open, and let $p\geq1$ be an integer. Consider a mapping $T:U\to\mathcal{L}(Y,E)$ such that
\[
G:U\times Y\to E,\qquad G(u,y):=T(u)y,
\]
is jointly of class $\mathcal{C}^p$. Then the following assertions are equivalent:
\begin{enumerate}
    \item $T\in\mathcal{C}^p(U,\mathcal{L}(Y,E))$.
    \item For every $u_0\in U$,
    \begin{equation}\label{eq:pointwise-equicontinuity}
    \lim_{u\to u_0}\sup_{\|y\|\leq1}
    \bigl\|D_1^pG(u,y)-D_1^pG(u_0,y)\bigr\|_{\mathcal{L}^p(X;E)}=0.
    \end{equation}
\end{enumerate}
Note that condition (ii) means that the family $\{D_1^pG(\cdot,y):\|y\|\leq1\}$ is equicontinuous at each $u_0\in U$.
\end{lemma}

\begin{proof}
Since
\[
T(u)=D_2G(u,0),
\]
we already have $T\in\mathcal{C}^{p-1}(U,\mathcal{L}(Y,E))$. For $0\leq j\leq p-1$, differentiation through the continuous linear evaluation at $y$ gives
\begin{equation}\label{eq:lower-operator-derivatives}
\bigl(D^jT(u)(h_1,\ldots,h_j)\bigr)y
=D_1^jG(u,y)(h_1,\ldots,h_j),
\end{equation}
with the usual interpretation when $j=0$.

For every $u\in U$, the mapping
\[
y\longmapsto D_1^pG(u,y)\in\mathcal{L}^p(X;E)
\]
is linear, by linearity of $G$ in its second variable, and continuous, by the joint $\mathcal{C}^p$ regularity of $G$. Hence it defines a bounded multilinear mapping
\[
B_p(u)\in\mathcal{L}^p\bigl(X;\mathcal{L}(Y,E)\bigr)
\]
through the identity
\begin{equation}\label{eq:top-operator-derivative-candidate}
\bigl(B_p(u)(h_1,\ldots,h_p)\bigr)y
=D_1^pG(u,y)(h_1,\ldots,h_p).
\end{equation}
Taking suprema over the unit vectors in each argument yields
\begin{equation}\label{eq:top-operator-norm-identity}
\|B_p(u)-B_p(u_0)\|
=\sup_{\|y\|\leq1}
\bigl\|D_1^pG(u,y)-D_1^pG(u_0,y)\bigr\|_{\mathcal{L}^p(X;E)}.
\end{equation}

Assume (ii). By \eqref{eq:top-operator-norm-identity}, $B_p$ is continuous in operator norm. Fix $u\in U$. For $h_p\in X$ sufficiently small, the segment $[u,u+h_p]$ is contained in $U$. For each fixed $y\in Y$, write $h=(h_1,\ldots,h_p)\in X^p$ and $h_{-p}:=(h_1,\ldots,h_{p-1})$, with the empty-tuple convention when $p=1$. The fundamental theorem of calculus applied to $D_1^{p-1}G(\cdot,y)$ gives
\begin{align*}
D_1^{p-1}G(u+h_p,y)h_{-p}-D_1^{p-1}G(u,y)h_{-p}
-D_1^pG(u,y)h=\int_0^1\bigl(D_1^pG(u+th_p,y)-D_1^pG(u,y)\bigr)h\,dt.
\end{align*}
Using \eqref{eq:lower-operator-derivatives} and \eqref{eq:top-operator-derivative-candidate}, and taking the supremum over $\|y\|\leq1$ and the remaining unit arguments, we obtain
\begin{align*}
\|D^{p-1}T(u+h_p)-D^{p-1}T(u)-B_p(u)(h_p,\cdot,\ldots,\cdot)\|&\leq\|h_p\|\sup_{0\leq t\leq1}\|B_p(u+th_p)-B_p(u)\|\\
&=o(\|h_p\|).
\end{align*}
When $p=1$, one reads $D^0T = T$ and omits the remaining arguments in these expressions. The last estimate follows from the continuity of $B_p$ at the fixed point $u$. Consequently, $D^{p-1}T$ is Fr\'echet differentiable and $D^pT=B_p$, under the natural identification of the corresponding spaces of multilinear mappings. Since $B_p$ is continuous, $T$ is of class $\mathcal{C}^p$.

Conversely, if (i) holds, differentiation of $G(u,y)=T(u)y$ with $y$ fixed gives $B_p(u)=D^pT(u)$. Therefore \eqref{eq:top-operator-norm-identity} and the continuity of $D^pT$ imply (ii). Finally, linearity in $y$ shows that the supremum over the unit ball equals that over the unit sphere when $Y\neq\{0\}$.
\end{proof}

The preceding lemma gives the announced characterization directly in terms of the squared-distance gradient in Poly--Raby coordinates. 

\begin{theorem}\label{thm:strong-characterization}
Let $M\subset\mathcal{H}$ be a closed set, let $k\geq2$, and let $\bar{x}\in M$. After translating $\bar{x}$ to $0$, the following assertions are equivalent:
\begin{enumerate}
\item $M$ is a $\mathcal{C}^k$-submanifold near $0$.
% \item $d_M^2$ is of class $\mathcal{C}^k$ near $0$, and the Poly--Raby coordinates $\Psi$ of Proposition~\ref{prop:PolyRaby-coordinates} can be chosen with image $V\times B_{Z^\perp}(0,\delta)$, where $V$ is an open neighborhood of $0$ in $Z=\rg(DP_M(0))$ and $\delta>0$, so that the mapping $g:V\times B_{Z^\perp}(0,\delta)\to Z$, defined by
% \[
% g(z,w):=-\frac12\pi_Z\nabla d_M^2\bigl(\Psi^{-1}(z,w)\bigr),
% \]
% satisfies, for some $0<\rho<\delta$ and every $z_0\in V$,
% \begin{equation}\label{eq:distance-pointwise-equicontinuity}
% \lim_{z\to z_0}\sup_{\substack{w\in Z^\perp\\\|w\|\leq\rho}}
% \bigl\|D_1^{k-1}g(z,w)-D_1^{k-1}g(z_0,w)\bigr\|_{\mathcal{L}^{k-1}(Z;Z)}=0.
% \end{equation}
\item $d_M^2$ is of class $\mathcal{C}^k$ near $0$ and  for the Poly--Raby coordinates $\Psi$ of Proposition~\ref{prop:PolyRaby-coordinates}, there exists $V$ neighborhood of $0$ in $Z=\rg(DP_M(0))$ and $\rho>0$ such that the function
\[
g(z,w):=-\frac12\pi_Z\nabla d_M^2\bigl(\Psi^{-1}(z,w)\bigr),
\]
is well-defined for $(z,w)\in V\times B_{Z^{\perp}}(0,\rho)$ and it verifies that
\begin{equation}\label{eq:distance-pointwise-equicontinuity}
\{D_1^{k-1}g(\cdot,w):w\in Z^\perp,\ \|w\|\leq\rho\}\text{ is equicontinuous at each }z_0\in V.
\end{equation}
\end{enumerate}
% In (ii), equicontinuity is taken with respect to the operator norm of $\mathcal{L}^{k-1}(Z;Z)$, meaning that for every $z_0\in V$ one has that
% \begin{equation*}
% \lim_{z\to z_0}\sup_{\substack{w\in Z^\perp\\\|w\|\leq\rho}}
% \bigl\|D_1^{k-1}g(z,w)-D_1^{k-1}g(z_0,w)\bigr\|_{\mathcal{L}^{k-1}(Z;Z)}=0.
% \end{equation*}
\end{theorem}

\begin{proof}
A $\mathcal{C}^k$-submanifold is weakly $\mathcal{C}^k$, so Theorem~\ref{thm:weak-characterization} gives $d_M^2\in\mathcal{C}^k$ under (i). It therefore suffices to work under this smoothness assumption and determine when the graph $\theta$ from Proposition~\ref{prop:PolyRaby-coordinates} is of class $\mathcal{C}^k$.

By Theorem~\ref{thm:weak-characterization} and its proof, $\theta$ is weakly $\mathcal{C}^k$. Identity~\eqref{eq:PolyRaby-adjoint} gives
\[
g(z,w)=D\theta(z)^*w.
\]
Thus $g$ extends linearly in $w$ to
\[
G:V\times Z^\perp\to Z,\qquad G(z,w)=D\theta(z)^*w,
\]
which is jointly of class $\mathcal{C}^{k-1}$ by Lemma~\ref{lem:joint-adjoint}. Apply Lemma~\ref{lem:pointwise-equicontinuity} with $p=k-1$ and $T(z)=D\theta(z)^*$. Linearity in $w$ shows that \eqref{eq:distance-pointwise-equicontinuity} is equivalent to
\[
T\in\mathcal{C}^{k-1}\bigl(V,\mathcal{L}(Z^\perp,Z)\bigr).
\]
Since taking the adjoint is a continuous linear isometry between $\mathcal{L}(Z,Z^\perp)$ and $\mathcal{L}(Z^\perp,Z)$, this is equivalent to $\theta\in\mathcal{C}^k(V,Z^\perp)$, and hence to (i). The same linearity proves that the condition is independent of the choice of $\rho\in(0,\delta)$.
\end{proof}

\section{Conclusions}\label{sec:conclusions}

In this work we have completed the extension of Poly and Raby's theorem to the infinite dimensions setting, by splitting the characterizations for smoothness of $d_M^2$ and smoothness of $M$ as submanifold. The final picture of smoothness characterizations for submanifolds in infinite dimensions is given by the following table:
\begin{table}[ht]
\centering
\renewcommand{\arraystretch}{1.25}
\begin{tabular}{cl}
\toprule
Regularity hypothesis on $d_M^2$ near $\bar{x}$ & Equivalent regularity of $M$ near $\bar{x}$ \\
\midrule
$d_M^2\in\mathcal{C}^{k}$ & $M$ is a weakly $\mathcal{C}^{k}$-submanifold \\
$d_M^2\in\mathcal{C}^{k}$ +~\eqref{eq:distance-pointwise-equicontinuity} & $M$ is a $\mathcal{C}^{k}$-submanifold \\
$d_M^2\in\mathcal{C}^{k,0}$ & $M$ is a $\mathcal{C}^{k,0}$-submanifold \\
$d_M^2\in\mathcal{C}^{k,1}$ & $M$ is a $\mathcal{C}^{k,1}$-submanifold \\
\bottomrule
\end{tabular}
\caption{Smoothness characterizations for submanifolds in Hilbert spaces, for $k\geq2$. Conditions 1 to 3 collapse for finite codimension of $M$.}
\label{tab:characs-full}
\end{table}

Locally, $\Psi^{-1}$ identifies the zero section $V\times\{0\}$ with $M$ and each normal-coordinate fiber $\{z\}\times B_{Z^\perp}(0,\delta)$ with a neighborhood of $(z,\theta(z))$ in its affine normal space $(z,\theta(z))+N_{(z,\theta(z))}M$. Using the graph representation $\theta$, one sees that $\Psi^{-1}(z,w)$ consists of the point $(z,\theta(z))$ plus the normal deviation $(-D\theta(z)^*w,w) \in N_{(z,\theta(z))}M$. 

The key point in Section~\ref{sec:PolyRaby-characterization} is that the $\mathcal{C}^k$-smoothness of $d_M^2$ is not directly related to the $\mathcal{C}^k$-smoothness of $\theta$, but rather to the $\mathcal{C}^{k-1}$-smoothness of the normal deviation, that is, to the joint $\mathcal{C}^{k-1}$-smoothness of $(z,w)\mapsto D\theta(z)^*w$. This is readily equivalent to the joint $\mathcal{C}^k$-smoothness of $(z,w)\mapsto \langle \theta(z),w\rangle$, which leads to the concept of weakly $\mathcal{C}^k$-submanifolds. 

In Section~\ref{sec:Pure-submanifolds}, condition~\eqref{eq:distance-pointwise-equicontinuity} can be interpreted as a pointwise equicontinuity condition on the last tangential derivative of $\pi_Z\circ \nabla d_M^2$, uniform with respect to the normal deviation $w\in Z^{\perp}$ in Poly--Raby coordinates. This condition is strictly stronger than plain $\mathcal{C}^k$-smoothness of $d_M^2$, as Section~\ref{sec:Counterexample} shows, and strictly weaker than local uniform continuity of $D^kd_M^2$, see \cite{Salas-Thibault:2019}.

\paragraph{Acknowledgments:} This research was supported by FONDECYT grant 1251159 and CMM BASAL fund FB210005 for center of excellence, ANID-Chile.

\paragraph{Use of AI:} This work was developed with GPT-6 Astra. Example~\ref{ex:Counterexample} was obtained directly with it. The example revealed the notion of weakly $\mathcal{C}^k$ functions, which was previously unknown to the author. With this insight, the rest of the article was quickly developed, also with the support of GPT-6 Astra. After the first draft, the manuscript was carefully revised, validated, and shortened by the author. Most of the references were included during this process.

\bibliographystyle{abbrvnat}
\bibliography{biblioSmoothDistance-2}

\rule{5 cm}{0.5 mm} \bigskip
    \newline\noindent David SALAS
	
	\medskip
	
	\noindent Instituto de Ciencias de la Ingenier\'{i}a, Universidad de
	O'Higgins\newline Av. Libertador Bernardo O'Higgins 611, Rancagua, Chile
	\smallskip
	
	\noindent E-mail: \texttt{david.salas@uoh.cl} \newline\noindent
	\texttt{http://davidsalasvidela.cl}, \\
    ORCID number: 0000-0003-3924-366X
\end{document}